\documentclass[11pt,final]{newarticle}
\usepackage{mathrsfs}
\usepackage{amssymb}
\usepackage{amsmath}

\usepackage{enumerate}
\usepackage{enumitem}
\setlist[enumerate,1]{label=(\arabic*).,ref=\thetheorem\,(\arabic*),
font=\textup,leftmargin=7mm,labelsep=1.5mm,topsep=0mm,itemsep=-0.8mm}
\setlist[enumerate,2]{label=(\alph*).,font=\textup,
leftmargin=7mm,labelsep=1.5mm,topsep=-0.8mm,itemsep=-0.8mm}

\usepackage{geometry}
\usepackage{microtype}

\newtheorem{theorem}{Theorem}[section]
\newtheorem{lemma}{Lemma}[section]
\newtheorem{corollary}{Corollary}[section]

\newdefinition{definition}{Definition}[section]
\newdefinition{remark}{Remark}[section]
\newdefinition{example}{Example}[section]
\newdefinition{claim}{Claim}

\newproof{proof}{Proof}

\makeatletter
\renewcommand*\env@matrix[1][\arraystretch]{%
\edef\arraystretch{#1}%
\hskip -\arraycolsep
\let\@ifnextchar\new@ifnextchar
\array{*\c@MaxMatrixCols c}}
\makeatother

\numberwithin{equation}{section}
\newcommand{\cupdot}{\dot{\cup}}

\renewcommand{\cup}{\operatorname*{\,\scalebox{0.8}{$\bigcup$}}\limits\,}
\allowdisplaybreaks

\begin{document}
\begin{frontmatter}
\title{On the irregularity of uniform hypergraphs\,\tnoteref{titlenote}}
\tnotetext[titlenote]{This work was supported by the National Nature
Science Foundation of China (Nos.\,11471210, 11571222)}

\author[address1]{Lele Liu}
\ead{ahhylau@gmail.com}

\author[address1]{Liying Kang}
\ead{lykang@shu.edu.cn}

\author[address1,address2]{Erfang Shan\corref{correspondingauthor}}
\cortext[correspondingauthor]{Corresponding author}
\ead{efshan@i.shu.edu.cn}

\address[address1]{Department of Mathematics, Shanghai University, Shanghai 200444, P.R. China}
\address[address2]{School of Management, Shanghai University, Shanghai 200444, P.R. China}

\begin{abstract}
Let $H$ be an $r$-uniform hypergraph on $n$ vertices and $m$ edges, and let
$d_i$ be the degree of $i\in V(H)$. Denote by $\varepsilon(H)$ the difference
of the spectral radius of $H$ and the average degree of $H$. Also, denote
\[
s(H)=\sum_{i\in V(H)}\left|d_i-\frac{rm}{n}\right|,~
v(H)=\frac{1}{n}\sum_{i\in V(H)}d_i^{\frac{r}{r-1}}-\left(\frac{rm}{n}\right)^{\frac{r}{r-1}}.
\]
In this paper, we investigate the irregularity of $r$-uniform hypergraph $H$ with
respect to $\varepsilon(H)$, $s(H)$ and $v(H)$, which extend relevant results
to uniform hypergraphs.
\end{abstract}

\end{frontmatter}

\section{Introduction}
Let $G=(V(G),E(G))$ be an undirected graph with $n$ vertices and $m$ edges without
loops and multiple edges, where $V(G)=[n]:=\{1,2,\ldots,n\}$. A graph $G$ is regular
if all its vertices have the same degree, otherwise it is irregular. In many
applications and problems it is of importance to know how irregular a given graph is.
Various measures of graph irregularity have been proposed and studied,  see, for example,
\cite{Bell,Collatz,Henning,Nikiforov:Degree deviation,Nikiforov2007} and references therein.

We first survey some known parameters used as measures of irregularity as well as their
respective properties. In 1957, Collatz and Sinogowitz \cite{Collatz} showed that the
spectral radius $\rho(G)$ of a graph $G$ is greater than or equal to the average degree
$\overline{d}(G)$, and the equality holds if and only if $G$ is regular. The fact allows
us to consider the difference $\varepsilon(G)=\rho(G)-\overline{d}(G)$ as a relevant
measure of irregularity of $G$. The authors also proved that, for $n\leq 5$, the
maximum value of $\varepsilon(G)$ is $\sqrt{n-1}-2+2/n$ and the maximal is attained for
the star $S_n$. Fifty years later, Aouchiche et al. \cite{Aouchiche2008} conjectured that
the most irregular connected graph on $n$ ($n\geq10$) vertices is a pineapple graph.
Recently, this conjecture was proved by Tait and Tobin \cite{Tait2016}. In 1992, Bell
\cite{Bell} suggested making the variance $v(G)$ of the vertex degrees of $G$ as a measure
of the irregularity, i.e.,
\[
v(G)=\frac{1}{n}\sum_{i=1}^nd_i^2-\left(\frac{2m}{n}\right)^2.
\]
The author compared $\varepsilon(G)$ and $v(G)$ for various classes of graphs, and showed
that they are not always compatible. Also, the most irregular graphs according to these
measures were determined for certain classes of graphs. In 2006, Nikiforov
\cite{Nikiforov:Degree deviation} introduced
\[
s(G)=\sum_{i\in V(G)}\left|d_i-\frac{2m}{n}\right|
\]
as a new measure of the irregularity of a graph $G$, and showed several inequalities with
respect to $\varepsilon(G)$, $s(G)$ and $v(G)$ as follows:
\begin{equation}
\label{eq:Nikiforov}
\frac{v(G)}{2\sqrt{2m}}\leq\rho(G)-\frac{2m}{n}\leq
\sqrt{s(G)}.
\end{equation}
In particular, for a bipartite graph $G$ with $m$ edges and partition $V(G)=V_1\cupdot V_2$,
Nikiforov \cite{Nikiforov:Degree deviation} defined
\[
s_2(G)=\sum_{i\in V_1} \left|d_i-\frac{m}{n_1}\right|+
\sum_{i\in V_2} \left|d_i-\frac{m}{n_2}\right|
\]
as a more relevant irregularity parameter than $s(G)$, where $n_1=|V_1|$, $n_2=|V_2|$. Also,
it was proved that
\begin{equation}
\label{eq:Nikiforov for bipartite}
\rho(G)-\frac{m}{\sqrt{n_1n_2}}\leq\sqrt{\frac{s_2(G)}{2}}.
\end{equation}
These irregularity measures as well as other attempts to measure the irregularity of a
graph were studied in several works
\cite{Dimitrov2014,Dimitrov2015,Edwards,Henning,Rautenbach}.

Our work in the present paper is to study the irregularity of uniform hypergraphs.
Denote by $\mathcal{H}(n,m)$ the set of all the $r$-uniform hypergraphs with $n$
vertices and $m$ edges. Let $H\in\mathcal{H}(n,m)$ be an $r$-uniform hypergraph, and
$\rho(H)$ be the spectral radius of $H$. In 2012, Cooper and Dutle
\cite{Cooper:Spectra Uniform Hypergraphs} showed that $\rho(H)\geq rm/n$. It is
clear that the equality holds if and only if $H$ is regular by \cite[Theorem 2]{Qi2013}.
Therefore, the value
\[
\varepsilon(H)=\rho(H)-\frac{rm}{n}
\]
can be viewed as a relevant measure of irregularity of $H$. Denote
\[
s(H)=\sum_{i\in V(H)}\left|d_i-\frac{rm}{n}\right|,
\]
where $d_i$ is the degree of vertex $i$ of $H$. Obviously, $s(H)\geq 0$, with equality
if and only if $H$ is regular. Analogous to the graph case, if $H\in\mathcal{H}(n,m)$
is an $r$-partite $r$-uniform hypergraph with partition $V(H)=V_1\cupdot V_2\cupdot\cdots\cupdot V_r$
and $|V_i|=n_i$, $i\in[r]$, we denote
\[
s_r(H)=\sum_{i\in[r]}\sum_{j\in V_i}\left|d_j-\frac{m}{n_i}\right|.
\]
For an $r$-uniform hypergraph $H\in\mathcal{H}(n,m)$, we also denote
\[
v(H)=\frac{1}{n}\sum_{i=1}^nd_i^{\frac{r}{r-1}}-\left(\frac{rm}{n}\right)^{\frac{r}{r-1}}.
\]
It follows from Power Mean inequality that $v(H)\geq 0$, with equality holds if and only if $H$ is
regular.

The main contribution of this paper is proposing some relations among $\varepsilon(H)$, $s(H)$
and $v(H)$, which extend relevant results to uniform hypergraphs. To be precise, we first generalize
\eqref{eq:Nikiforov for bipartite} to $r$-partite $r$-uniform hypergraphs as follows.
\begin{theorem}
\label{thm:Main result-1}
Let $H\in\mathcal{H}(n,m)$ be an $r$-partite $r$-uniform hypergraph with partition
$V(H)=V_1\cupdot V_2\cupdot\cdots\cupdot V_r$. Let $|V_i|=n_i$, $i\in[r]$. Then
\[
\rho(H)-\frac{m}{\sqrt[r]{n_1n_2\cdots n_r}}\leq
\left(\frac{s_r(H)}{2}\right)^{\frac{r-1}{r}}.
\]
\end{theorem}
The main frame of \autoref{thm:Main result-1} is inspired by that of \cite{Nikiforov:Degree deviation}.
By virtue of \autoref{thm:Main result-1} and the direct product operation of hypergraphs,
we obtain the following result concerning $\varepsilon(H)$, $s(H)$ and $v(H)$, which
generalize the result \eqref{eq:Nikiforov}.
\begin{theorem}
\label{thm:Main result-2}
Let $H\in\mathcal{H}(n,m)$. Then
\[
\frac{r-1}{\sqrt[r]{m}}\left(\frac{\sqrt[r]{r!}}{r^r}\right)^{\frac{1}{r-1}}v(H)
\leq\rho(H)-\frac{rm}{n}
\leq\frac{r}{\sqrt[r]{r!}}\left(\frac{s(H)}{2}\right)^{\frac{r-1}{r}}.
\]
\end{theorem}

\section{Preliminaries}
\label{sec2}
In this section, we first present some necessarily notions and definitions of hypergraphs
and tensors which will be used in the sequel.

A hypergraph $H =(V(H),E(H))$ is a pair consisting of a vertex set $V(H)=[n]$, and a
set $E(H)=\{e_1,e_2,\ldots,e_m\}$ of subsets of $V(H)$, the edges of $H$. For a vertex
$i\in V(H)$, the {\em degree} of $i$, denoted by $d_H(i)$ or simply by $d_i$, is the
number of edges containing $i$. A hypergraph is called {\em regular} if all its vertices
have the same degree, otherwise it is {\em irregular}. The minimum and maximum degrees among
the vertices of $H$ are denoted by $\delta(H)$ and $\Delta(H)$, respectively. An $r$-uniform
hypergraph $H$ is called $k$-{\em partite} if its vertex set $V(H)$ can be partitioned into
$k$ sets such that each edge contains at most one vertex from each set. An edge maximal
$k$-partite $r$-uniform hypergraph is called {\em complete $k$-partite}.

Let $H_1$ and $H_2$ be two $r$-uniform hypergraphs. Denote by $H_1\cup H_2$ the {\em union} of
$H_1$ and $H_2$, i.e., $V(H_1\cup H_2)=V(H_1)\cup V(H_2)$ and $E(H_1\cup H_2)=E(H_1)\cup E(H_2)$.
The {\em direct product} $H_1\times H_2$ of $H_1$ and $H_2$ is defined as an $r$-uniform
hypergraph with vertex set $V(H_1\times H_2) = V(H_1)\times V(H_2)$, and
$\{(i_1,j_1),(i_2,j_2),\ldots,(i_r,j_r)\}\in E(H_1\times H_2)$ if and only if
$\{i_1,i_2,\ldots,i_r\}\in E(H_1)$ and $\{j_1,j_2,\ldots,j_r\}\in E(H_2)$.

For positive integers $r$ and $n$, a real {\em tensor} $\mathcal{A}=(a_{i_1i_2\cdots i_r})$ of
order $r$ and dimension $n$ refers to a multidimensional array (also called {\em hypermatrix})
with entries $a_{i_1i_2\cdots i_r}$ such that $a_{i_1i_2\cdots i_r}\in\mathbb{R}$ for all
$i_1$, $i_2$, $\ldots$, $i_r\in[n]$. The following general product of tensors was defined by
Shao \cite{Shao:General product}, which is a generalization of the matrix case.

\begin{definition}[\cite{Shao:General product}]
\label{defn:General product}
Let $\mathcal{A}$ (and $\mathcal{B}$) be an order $r\geq 2$ (and order $k\geq 1$),
dimension $n$ tensor. Define the product $\mathcal{AB}$ to be the following tensor $\mathcal{C}$
of order $(r-1)(k-1)+1$ and dimension $n$
\[
c_{i\alpha_1\cdots\alpha_{r-1}}=\sum_{i_2,\ldots,i_r=1}^na_{ii_2\cdots i_r}
b_{i_2\alpha_1}\cdots b_{i_r\alpha_{r-1}}
~~(i\in [n], \alpha_1,\ldots,\alpha_{r-1}\in [n]^{k-1}).
\]
\end{definition}
From the above definition, if $x=(x_1,x_2,\ldots,x_n)^{\mathrm{T}}$ is a vector, we have
\begin{equation}
\label{eq:Ax equation}
(\mathcal{A}x)_i=\sum_{i_2,\ldots,i_r=1}^na_{ii_2\cdots i_r}x_{i_2}\cdots x_{i_r},~~
i\in [n].
\end{equation}

In 2005, Qi \cite{Qi2005} and Lim \cite{Lim} independently introduced the definition of
eigenvalues of a tensor. Let $\mathcal{A}$ be an order $r$ dimension $n$ tensor,
$x=(x_1,x_2,\ldots,x_n)^{\mathrm{T}}\in\mathbb{C}^n$ be a column vector of dimension $n$.
If there exists a number $\lambda\in\mathbb{C}$ and a nonzero vector $x\in\mathbb{C}^{n}$
such that
\[
\mathcal{A}x=\lambda x^{[r-1]},
\]
then $\lambda$ is called an {\em eigenvalue} of $\mathcal{A}$, $x$ is called an
{\em eigenvector} of $\mathcal{A}$ corresponding to the eigenvalue $\lambda$,
where $x^{[r-1]}$ is the Hadamard power of $x$, i.e.,
$x^{[r-1]}=(x_{1}^{r-1},x_2^{r-1},\ldots,x_{n}^{r-1})^{\mathrm{T}}$. The {\em spectral
radius} of $\mathcal{A}$, denoted by $\rho(\mathcal{A})$, is the maximum modulus of the
eigenvalues of $\mathcal{A}$.

In 2012, Cooper and Dutle \cite{Cooper:Spectra Uniform Hypergraphs}
defined the adjacency tensors $\mathcal{A}(H)$ for an $r$-uniform hypergraphs $H$.

\begin{definition}
[\cite{Cooper:Spectra Uniform Hypergraphs}]
Let $H=(V(H),E(H))$ be an $r$-uniform hypergraph on $n$ vertices. The adjacency
tensor of $H$ is defined as the order $r$ and dimension $n$ tensor
$\mathcal{A}(H)=(a_{i_1i_2\cdots i_r})$, whose $(i_1i_2\cdots i_r)$-entry is
\[
a_{i_1i_2\cdots i_r}=\begin{cases}
\frac{1}{(r-1)!}, & \text{if}~\{i_1,i_2,\ldots,i_r\}\in E(H),\\
0, & \text{otherwise}.
\end{cases}
\]
\end{definition}

For an $r$-uniform hypergraph $H$, the spectral radius of $H$, denoted by $\rho(H)$,
is defined to be that of its adjacency tensor $\mathcal{A}(H)$. In general, an
$r$-uniform hypergraph $H$ can be decomposed into components $H_i=(V(H_i),E(H_i))$
for $i=1$, $2$, $\ldots$, $s$. Denote the spectral radii of $H$ and $H_i$ by $\rho(H)$
and $\rho(H_i)$, respectively.  Theorem 3.3 in \cite{Qi2014} implies that
\[
\rho(H)=\max_{1\leq i\leq s}\{\rho(H_i)\}.
\]

Friedland et al. \cite{Friedland2013} defined the weak irreducibility of a nonnegative tensor
by using the strong connectivity of a graph associated to the nonnegative tensor. Later,
Yang et al. \cite{Yang2011-2} presented an equivalent definition of the weak irreducibility
from the algebraic point of view.
\begin{definition}[\cite{Yang2011-2}]
Let $\mathcal{A}$ be an order $r$ dimension $n$ tensor. If there exists a nonempty
proper index subset $I\subseteq [n]$ such that
\[
a_{i_1i_2\cdots i_r}=0~~(\forall~i_1\in I,~\text{and at least one of}~i_2,\ldots,i_r\notin I).
\]
Then $\mathcal{A}$ is called weakly reducible. If $\mathcal{A}$ is not weakly reducible,
then $\mathcal{A}$ is called weakly irreducible.
\end{definition}

It was proved that an $r$-uniform hypergraph $H$ is connected if and only if its adjacency
tensor $\mathcal{A}(H)$ is weakly irreducible (see \cite{Pearson2014}).

Let $\mathcal{A}=(a_{i_1i_2\cdots i_r})$ be a nonnegative tensor of order $r$ and dimension $n$.
For any $i\in [n]$, we write
\[
r_i(\mathcal{A})=\sum_{i_2,\ldots,i_r=1}^na_{ii_2\cdots i_r}.
\]
The following bound for $\rho(\mathcal{A})$ in terms of $r_i(\mathcal{A})$ was proposed
in \cite{Yang2010}, and the conditions for the equal cases were studied in \cite{Fan2015}.

\begin{lemma}[\cite{Fan2015,Yang2010}]
\label{lem:rho(A) upper bound}
Let $\mathcal{A}$ be a nonnegative tensor of order $r$ and dimension $n$. Then
\begin{equation}
\label{eq:r_i equality}
\min_{1\leq i\leq n}r_i(\mathcal{A})\leq\rho(\mathcal{A})\leq
\max_{1\leq i\leq n}r_i(\mathcal{A}).
\end{equation}
Moreover, if $\mathcal{A}$ is weakly irreducible, then one of the equalities in
\eqref{eq:r_i equality} holds if and only if
$r_1(\mathcal{A})=r_2(\mathcal{A})=\cdots=r_n(\mathcal{A})$.
\end{lemma}

\begin{lemma}[\cite{Shao:General product,Yang2011}]
\label{lem:Same spectra}
Let $\mathcal{A}$ and $\mathcal{B}$ be two order $r$ dimension $n$ tensors.
If there is a nonsingular diagonal matrix $P$ of order $n$ such that
$\mathcal{B}=P^{-(r-1)}\mathcal{A}P$, then $\mathcal{A}$ and $\mathcal{B}$
have the same eigenvalues.	
\end{lemma}

\begin{remark}
Let $P=\text{diag}\{p_1,p_2,\ldots,p_n\}$ be a nonsingular diagonal matrix,
and $\mathcal{A}(H)=(a_{i_1i_2\cdots i_r})$ be the adjacency tensor of an
$r$-uniform hypergraph $H$. According to \autoref{defn:General product}, we have
\begin{equation}
\label{eq:PAP}
(P^{-(r-1)}\mathcal{A}(H)P)_{i_1i_2\cdots i_r}=
p_{i_1}^{-(r-1)}a_{i_1i_2\cdots i_r}p_{i_2}\cdots p_{i_r}.
\end{equation}
\end{remark}

\begin{lemma}[\cite{Nikiforov}]
\label{lem:rho<m}
Let $H\in\mathcal{H}(n,m)$ be an $r$-uniform hypergraph. Then
\[
\rho(H)\leq\frac{r}{\sqrt[r]{r!}}m^{\frac{r-1}{r}}.
\]
Moreover, if $H$ is $r$-partite, then
\[
\rho(H)\leq m^{\frac{r-1}{r}},
\]
equality holds if and only if $H$ is complete $r$-partite.
\end{lemma}

The Weyl type inequality for uniform hypergraphs is stated as follows.
\begin{lemma}[\cite{Nikiforov}]
\label{lem:Weyl's inequality}
Let $H_1$ and $H_2$ be $r$-uniform hypergraphs. Then
\[
\rho(H_1\cup H_2)\leq \rho(H_1)+\rho(H_2).
\]
\end{lemma}

\section{Irregularity of uniform hypergraphs}
In this section, we shall prove \autoref{thm:Main result-1} and \autoref{thm:Main result-2}.
Before continuing, we present an upper bound for the spectral radius of an $r$-uniform
hypergraph, which generalizes a result in \cite{Berman2001}. It is noted that
the same result has been proved by Nikiforov \cite{Nikiforov2017}.
Here we add a characterization for the equality.

\begin{lemma}[\cite{Nikiforov2017}]\label{lem:r-partite}
Suppose that $H$ is a connected $r$-uniform hypergraph on $n$ vertices. Then
\begin{equation}
\label{eq:rho(H)<sqrt[r]}
\rho(H)\leq\max_{\{i_1,i_2,\ldots,i_r\}\in E(H)}
\left\{\sqrt[r]{d_{i_1}d_{i_2}\cdots d_{i_r}}\right\},
\end{equation}
with equality holds if and only if $d_{i_1}d_{i_2}\cdots d_{i_r}$ is a constant
for any $\{i_1,i_2,\ldots,i_r\}\in E(H)$.
\end{lemma}

\begin{proof}
Let $P=\text{diag}\,\{p_1,p_2,\ldots,p_n\}$ be a nonsingular diagonal matrix.
By \eqref{eq:PAP} we have
\begin{align*}
r_i(P^{-(r-1)}\mathcal{A}(H)P) & =
\sum_{i_2,\ldots,i_r=1}^np_i^{-(r-1)}a_{ii_2\cdots i_r}p_{i_2}\cdots p_{i_r}\\
& =\sum_{\{i,i_2,\ldots,i_r\}\in E(H)}p_i^{-(r-1)}p_{i_2}\cdots p_{i_r}.
\end{align*}
Setting $p_i=\sqrt[r]{d_i}$, $i\in [n]$, we see
\begin{align}
r_i(P^{-(r-1)}\mathcal{A}(H)P) & =
\sum_{\{i,i_2,\ldots,i_r\}\in E(H)}\frac{\sqrt[r]{d_id_{i_2}\cdots d_{i_r}}}{d_i} \label{eq:r_i}\\
& \leq\max_{\{i_1,i_2,\ldots,i_r\}\in E(H)}\left\{\sqrt[r]{d_{i_1}d_{i_2}\cdots d_{i_r}}\right\}
\label{eq:rho(H)<max}.
\end{align}
By \autoref{lem:rho(A) upper bound} and \autoref{lem:Same spectra}, we deduce that
\begin{equation}
\label{eq:rho(H)<max1}
\rho(H)=\rho(P^{-(r-1)}\mathcal{A}(H)P) \leq
\max_{1\leq i\leq n}\left\{r_i(P^{-(r-1)}\mathcal{A}(H)P)\right\}.
\end{equation}
Then \eqref{eq:rho(H)<max} and \eqref{eq:rho(H)<max1} imply that
\[
\rho(H)\leq\max_{\{i_1,i_2,\ldots,i_r\}\in E(H)}\left\{\sqrt[r]{d_{i_1}d_{i_2}\cdots d_{i_r}}\right\}.
\]

If the equality in \eqref{eq:rho(H)<sqrt[r]} holds, then the equality in \eqref{eq:rho(H)<max1}
holds. Since $H$ is connected, $\mathcal{A}(H)$ is weakly irreducible. Therefore,
$P^{-(r-1)}\mathcal{A}(H)P$ is also weakly irreducible. By \autoref{lem:rho(A) upper bound},
$r_i(P^{-(r-1)}\mathcal{A}(H)P)$ is a constant, $i\in [n]$. Furthermore, the equality
in \eqref{eq:rho(H)<max} holds. So, $d_{i_1}d_{i_2}\cdots d_{i_r}\equiv c$ is a
constant for any $\{i_1,i_2,\ldots,i_r\}\in E(H)$. Conversely, assume that for any
$\{i_1,i_2,\ldots,i_r\}\in E(H)$, $d_{i_1}d_{i_2}\cdots d_{i_r}\equiv c$ is a constant.
It follows from \autoref{lem:rho(A) upper bound} and \eqref{eq:r_i} that
\[
\sqrt[r]{c}=\min_{1\leq i\leq n}\left\{r_i(P^{-(r-1)}\mathcal{A}(H)P)\right\}
\leq\rho(H)\leq\max_{1\leq i\leq n}
\left\{r_i(P^{-(r-1)}\mathcal{A}(H)P)\right\}=\sqrt[r]{c},
\]
which yields that $\rho(H)=\sqrt[r]{c}$, as desired.
\end{proof}

\begin{remark}
We now consider a similar topic as \autoref{lem:r-partite} which is of independent
interest. Suppose that $H$ is a connected $r$-uniform hypergraph on $n$ vertices. Let
$x=\frac{1}{\sqrt[r]{rm}}(\sqrt[r]{d_1},\sqrt[r]{d_2},\ldots,\sqrt[r]{d_n})^{\mathrm{T}}$
be a column vector. By \cite[Theorem 2]{Qi2013} and AM--GM inequality, we have
\begin{align*}
\rho(H) & \geq x^{\mathrm{T}}(\mathcal{A}x)=\frac{1}{m}\sum_{\{i_1,i_2,\ldots,i_r\}\in E(H)}
\sqrt[r]{d_{i_1}d_{i_2}\cdots d_{i_r}}\\
& \geq\Bigg(\prod_{\{i_1,i_2,\ldots,i_r\}\in E(H)}
\sqrt[r]{d_{i_1}d_{i_2}\cdots d_{i_r}}\Bigg)^{\frac{1}{m}}.
\end{align*}
It follows that
\[
\rho(H)^r\geq\Bigg(\prod_{\{i_1,i_2,\ldots,i_r\}\in E(H)}d_{i_1}d_{i_2}\cdots d_{i_r}\Bigg)^{\frac{1}{m}}.
\]
From the inequality between geometric and harmonic means, we obtain
\begin{equation}
\label{eq:GM-HM}
\rho(H)^r\geq\frac{m}{\displaystyle\sum_{\{i_1,i_2,\ldots,i_r\}\in E(H)}\frac{1}{d_{i_1}d_{i_2}\cdots d_{i_r}}}.
\end{equation}
Clearly, equality in \eqref{eq:GM-HM} holds if and only if $d_{i_1}d_{i_2}\cdots d_{i_r}$
is a constant for any $\{i_1,i_2,\ldots,i_r\}\in E(H)$. The above inequality generalize a
result in \cite{HoffmanWolfe1995} (see also \cite{SimicStevanovic2003}).
\end{remark}

The following lemma is needed, and the arguments have been used in
\cite{Nikiforov:Degree deviation}.
\begin{lemma}\label{lem:Regular}
Let $H\in\mathcal{H}(n,m)$ be an $r$-uniform hypergraph. Then there exists an
$r$-uniform hypergraph $\widehat{H}\in\mathcal{H}(n,m)$ such that
$\Delta(\widehat{H})-\delta(\widehat{H})\leq 1$
and $\widehat{H}$ differs from $H$ in at most $s(H)$ edges.
\end{lemma}

\begin{proof}
Denote $d=\left\lfloor rm/n\right\rfloor$ for short. We first show that for
any $H\in\mathcal{H}(n,m)$, there exists $H^*\in\mathcal{H}(n,m)$ such that
either $\delta(H^*)=d$ or $\Delta(H^*)=d+1$. If $\delta(H)\leq d-1$ and
$\Delta(H)\geq d+2$, then we select $i$, $j\in V(H)$ such that $d_i=\delta(H)$
and $d_j=\Delta(H)$. Since $d_j>d_i$, there exists an edge $e\in E(H)$ such
that $j\in e$, $i\notin e$ and $e':=(e\backslash\{j\})\cup\{i\}\notin E(H)$.
Denote $H':=H-e+e'$. Clearly, $H'\in\mathcal{H}(n,m)$ and $H'$ differs from
$H$ in two edges. Moreover, we have
\begin{align*}
s(H){-}s(H') & {=} \left|d_i{-}\frac{rm}{n}\right|{+}\left|d_j{-}\frac{rm}{n}\right|
{-}\left|(d_i{+}1){-}\frac{rm}{n}\right|{-}\left|(d_j{-}1){-}\frac{rm}{n}\right|\\
& {=} \left(\frac{rm}{n}{-}d_i\right)\!\!+\!\!\left(d_j{-}\frac{rm}{n}\right)
\!\!-\!\!\left(\frac{rm}{n}{-}(d_i{+}1)\right)\!\!-\!\!\left((d_j{-}1){-}\frac{rm}{n}\right)\!=\!2.
\end{align*}
Repeating the above process, we can get an $r$-uniform hypergraph $H^*\in\mathcal{H}(n,m)$
such that either $\delta(H^*)=d$ or $\Delta(H^*)=d+1$, and $H^*$ differs from $H$ in
$(s(H)-s(H^*))$ edges.

Without loss of generality, we may assume that $\delta(H^*)=d$ (the other case can be
proved similarly). If $\Delta(H^*)\leq d+1$, then $\widehat{H}=H^*$ is the desired
hypergraph. Otherwise, assume that $\Delta(H^*)\geq d+2$. Denote
\begin{align*}
A & =\{i\in V(H^*)\,|\,d_{H^*}(i)=d\}, \\
B & =\{i\in V(H^*)\,|\,d_{H^*}(i)=d+1\},\\
C & =\{i\in V(H^*)\,|\,d_{H^*}(i)\geq d+2\},
\end{align*}
and $|A|=k$, $|B|=s$. Let $i\in A$, $j\in C$ with $d_{H^*}(j)=\Delta(H^*)$. Notice that
$d_{H^*}(j)>d$, then there is an edge $e\in E(H^*)$ such that $j\in e$, $i\notin e$ and
$e'':=(e\backslash\{j\})\cup\{i\}\notin E(H^*)$. Let $H'':=H^*-e+e''$. Then
$H''\in\mathcal{H}(n,m)$ and $H''$ differs from $H^*$ in two edges. Repeating the process at most
$\ell:=\sum_{u\in C}(d_{H^*}(u)-d-1)$ times, we can obtain the desired $r$-uniform
hypergraph $\widehat{H}\in\mathcal{H}(n,m)$. Therefore, $\widehat{H}$ differs $H$ at most
$(s(H)-s(H^*)+2\ell)$ edges.

In the following, we will show that $s(H)-s(H^*)+2\ell\leq s(H)$. Consider the $r$-uniform
hypergraph $H^*$, we have
\begin{align*}
\frac{rm}{n} & =\frac{1}{n}
\Bigg(\sum_{i\in A}d_{H^*}(i)+\sum_{i\in B}d_{H^*}(i)+\sum_{i\in C}d_{H^*}(i)\Bigg)\\
& =\frac1n\Bigg(kd+s(d+1)+\sum_{i\in C}d_{H^*}(i)\Bigg)\\
& =\frac1n[kd+s(d+1)+(n-k-s)(d+1)+\ell]\\
& =d+1+\frac{\ell-k}{n}.
\end{align*}
Recall that $d=\lfloor rm/n\rfloor$. Hence $\ell<k$. Furthermore,
\begin{align*}
s(H^*) & =\sum_{i\in A}\left|d_{H^*}(i)-\frac{rm}{n}\right|
+\sum_{i\in B}\left|d_{H^*}(i)-\frac{rm}{n}\right|
+\sum_{i\in C}\left|d_{H^*}(i)-\frac{rm}{n}\right|\\
& =k\left(\frac{rm}{n}-d\right)+s\left(d+1-\frac{rm}{n}\right)+
\sum_{i\in C}\left(d_{H^*}(i)-\frac{rm}{n}\right)\\
& =k\left(1+\frac{\ell-k}{n}\right)+s\cdot\frac{k-\ell}{n}+\ell+
(n-k-s)\cdot\frac{k-\ell}{n}\\
& =2k\left(1-\frac{k-\ell}{n}\right)+2\ell>2\ell,
\end{align*}
then the result follows.
\end{proof}

By applying \autoref{lem:Regular} to each vertex class of an $r$-partite $r$-uniform
hypergraph, we can obtain the following corollary.

\begin{corollary}\label{coro:Regular}
Let $H\in\mathcal{H}(n,m)$ be an $r$-partite $r$-uniform hypergraph. Then
there exists an $r$-partite $r$-uniform hypergraph $\widehat{H}$ such that
$|d_{\widehat{H}}(i)-d_{\widehat{H}}(j)|\leq 1$ for any $i$, $j$ belonging
to the same vertex class and $\widehat{H}$ differs from $H$ in at most $s_r(H)$
edges.
\end{corollary}

In the sequel, we shall prove \autoref{thm:Main result-1}.
For this purpose, we need the following concept.
Let $H$ be an $r$-uniform hypergraph on $n$ vertices and $k_1$, $k_2$, $\ldots$,
$k_n$ be positive integers. Denote $H(k_1,k_2,\ldots,k_n)$ for the $r$-uniform
hypergraph obtained by replacing each vertex $i\in V(H)$ with a set $U_i$
of size $k_i$ and each edge $\{i_1,i_2,\ldots,i_r\}\in E(H)$ with a complete
$r$-partite $r$-uniform hypergraph with vertex classes $U_{i_1}$, $U_{i_1}$,
$\ldots$, $U_{i_r}$. The hypergraph $H(k_1,k_2,\ldots,k_n)$ is called a
{\em blow-up} of $H$.

\begin{lemma}[\cite{Nikiforov}]
\label{lem:Blown-up}
Let $H$ be an $r$-uniform hypergraph on $n$ vertices. Then
\[
\rho(H(k,k,\ldots,k))=k^{r-1}\rho(H).
\]
\end{lemma}

\noindent{\bfseries Proof of \autoref{thm:Main result-1}.}
By \autoref{coro:Regular}, there exists an $r$-partite $r$-uniform hypergraph
$\widehat{H}\in\mathcal{H}(n,m)$ such that $|d_{\widehat{H}}(i)-d_{\widehat{H}}(j)|\leq 1$
for any $i$, $j$ belonging to the same vertex class and $\widehat{H}$ differs
from $H$ in at most $s_r(H)$ edges. Therefore $2|E(H)\backslash E(\widehat{H})|\leq s_r(H)$.
We need the following two claims.
\begin{claim}
\label{claim1}
$\displaystyle\sqrt[r]{n_1n_2\cdots n_r}\geq\sqrt[r]{n/r}$.
\end{claim}

\noindent{\bf Proof of Claim 1.}
Without loss of generality, we assume that $n_1-n_2\geq0$. We will replace the
pair $n_1$ and $n_2$ by $n_1'=n_1+1$ and $n_2'=n_2-1$. Notice that $n_1'$ and $n_2'$
have the same sum as $n_1$ and $n_2$ while decreasing the product. To be precise,
$n_1'n_2'=(n_1+1)(n_2-1)<n_1n_2$, and therefore $(n_1'n_2')n_3\cdots n_r<n_1n_2n_3\cdots n_r$.
Repeating this process, we know that $\sqrt[r]{n_1n_2\cdots n_r}$ attaining the minimum when
one of $n_1$, $n_2$, $\ldots$, $n_r$ is $(n-r+1)$ and the others are $1$. It follows that
$\sqrt[r]{n_1n_2\cdots n_r}\geq\sqrt[r]{n-r+1}\geq\sqrt[r]{n/r}$.
The proof of the claim is completed.

\begin{claim}
\label{claim2}
$\displaystyle\rho(\widehat{H})\leq
\frac{m}{\sqrt[r]{n_1n_2\cdots n_r}}+\left(\frac{n}{r}\right)^{1-\frac{1}{r}}$.
\end{claim}

\noindent{\bf Proof of Claim 2.}
Let $\Delta_i=\max\{d_j\,|\,j\in V_i(\widehat{H})\}$, $i\in[r]$, where
$V_1(\widehat{H})$, $V_2(\widehat{H})$, $\ldots$, $V_r(\widehat{H})$
are the vertex classes of $\widehat{H}$. Hence $\Delta_i\leq m/n_i+1$
by \autoref{coro:Regular} and the proof of \autoref{lem:Regular}.
Using \autoref{lem:r-partite} gives
\[
\rho(\widehat{H})\leq\sqrt[r]{\Delta_1\Delta_2\cdots\Delta_r}
\leq\sqrt[\leftroot{-2}\uproot{14}r]{\left(\frac{m}{n_1}+1\right)
\left(\frac{m}{n_2}+1\right)\cdots\left(\frac{m}{n_r}+1\right)}.
\]
It suffices to show that
\[
\prod_{i=1}^r\left(\frac{m}{n_i}+1\right)\leq
\left(\frac{m}{\sqrt[r]{n_1n_2\cdots n_r}}+\left(\frac{n}{r}\right)^{1-\frac{1}{r}}\right)^r.
\]
Denote by $e_0(n_1,n_2,\ldots,n_r)=1$ and
\[
e_j(n_1,n_2,\ldots,n_r)=
\sum_{1\leq i_1<i_2<\cdots<i_j\leq r}n_{i_1}n_{i_2}\cdots n_{i_j},~j=1,2,\ldots,r,
\]
the $j$-th elementary symmetric polynomials in $n_1$, $n_2$, $\ldots$, $n_r$.
By \autoref{claim1} and Maclaurin's inequality, for any $ i\in[r]$ we have
\[
\left(\frac{n}{r}\right)^{1-\frac{1}{r}}\sqrt[r]{n_1n_2\cdots n_r}\geq\frac{n}{r}=
\frac{e_1(n_1,n_2,\ldots,n_r)}{\binom{r}{1}}\geq
\left(\frac{e_i(n_1,n_2,\ldots,n_r)}{\binom{r}{i}}\right)^{\frac{1}{i}},
\]
which yields that
\[
e_i(n_1,n_2,\ldots,n_r)\leq\binom{r}{i}
\left(\frac{n}{r}\right)^{\frac{(r-1)i}{r}}(n_1n_2\cdots n_r)^{\frac{i}{r}}.
\]
Notice that
\[
\prod_{i=1}^r\left(\frac{m}{n_i}+1\right)=
\sum_{i=0}^r\frac{e_i(n_1,n_2,\ldots,n_r)}{n_1n_2\cdots n_r}\cdot m^{r-i}.
\]
Therefore, we obtain that
\begin{align*}
\prod_{i=1}^r\left(\frac{m}{n_i}+1\right) & \leq
\sum_{i=0}^r\binom{r}{i}\left(\frac{m}{\sqrt[r]{n_1n_2\cdots n_r}}\right)^{r-i}
\left(\frac{n}{r}\right)^{\frac{(r-1)i}{r}}\\
& =\left(\frac{m}{\sqrt[r]{n_1n_2\cdots n_r}}+\left(\frac{n}{r}\right)^{1-\frac{1}{r}}\right)^r.
\end{align*}
The proof of the claim is completed.

We will take the proof technique from \cite{Nikiforov:Degree deviation}. Let
$H_1=(V(H),E(H)\cup E(\widehat{H}))$ and $H_2=(V(H),E(H)\backslash E(\widehat{H}))$.
Clearly, $H$ is a subhypergraph of $H_1$, then $\rho(H)\leq\rho(H_1)$.
Therefore, by \autoref{lem:Weyl's inequality}, we have
\[
\rho(H)\leq\rho(H_1)=\rho(H_2\cup\widehat{H})\leq\rho(H_2)+\rho(\widehat{H}).
\]
It follows from \autoref{lem:rho<m} that
\begin{equation}
\label{eq:rho(H)-rho(H')}
\rho(H)-\rho(\widehat{H}) \leq\rho(H_2)\leq(E(H)\backslash E(\widehat{H}))^{\frac{r-1}{r}}
\leq\left(\frac{s_r(H)}{2}\right)^{\frac{r-1}{r}}.
\end{equation}
Finally, by \eqref{eq:rho(H)-rho(H')} and \autoref{claim2}, we obtain
\begin{equation}
\label{eq:Finally}
\rho(H)-\frac{m}{\sqrt[r]{n_1n_2\cdots n_r}}\leq
\left(\frac{s_r(H)}{2}\right)^{\frac{r-1}{r}}+\left(\frac{n}{r}\right)^{1-\frac{1}{r}}.
\end{equation}

Let $H(k,k,\ldots,k)$ be a blown-up of $H$. Clearly,
\[
|V(H(k,k,\ldots,k))|=kn,~|E(H(k,k,\ldots,k))|=k^rm.
\]
Applying \eqref{eq:Finally} for $H(k,k,\ldots,k)$, we have
\begin{align*}
\rho(H(k,k,\ldots,k))-\frac{k^rm}{\sqrt[r]{(kn_1)(kn_2)\cdots (kn_r)}} & \leq
\left(\frac{s_r(H(k,k,\ldots,k))}{2}\right)^{\frac{r-1}{r}}+\left(\frac{kn}{r}\right)^{1-\frac{1}{r}}.
\end{align*}
On the other hand, notice that $\rho(H(k,k,\ldots,k))=k^{r-1}\rho(H)$ by
\autoref{lem:Blown-up} and
\begin{align*}
s_r(H(k,k,\ldots,k))
& =k\sum_{i\in[r]}\sum_{j\in V_i}\left|d_{H(k,\ldots,k)}(j)-\frac{k^rm}{kn_i}\right|\\
& =k^r\sum_{i\in[r]}\sum_{j\in V_i}\left|d_{H}(j)-\frac{m}{n_i}\right|\\
& =k^rs_r(H),
\end{align*}
which follows that
\[
k^{r-1}\rho(H)-\frac{k^{r-1}m}{\sqrt[r]{n_1n_2\cdots n_r}}
\leq\left(\frac{k^rs_r(H)}{2}\right)^{\frac{r-1}{r}}+\left(\frac{kn}{r}\right)^{1-\frac{1}{r}}.
\]
Therefore, we obtain
\[
\rho(H)-\frac{m}{\sqrt[r]{n_1n_2\cdots n_r}}\leq\left(\frac{s_r(H)}{2}\right)^{\frac{r-1}{r}}+
\left(\frac{n}{r}\right)^{1-\frac{1}{r}}\cdot\frac{1}{k^{r+\frac{1}{r}-2}}.
\]
Take the limit $k\to +\infty$ on both sides of the above equation, we obtain the desired result.

The proof is completed. \qed \vspace{3mm}

In the following we will give a proof of \autoref{thm:Main result-2} in virtue of
\autoref{thm:Main result-1} and the following result.

\begin{lemma}[\cite{Liu:Bounds concerning degrees}]
\label{lem:rho bound concerning degrees}
Suppose that $H$ is an $r$-uniform hypergraph on $n$ vertices. Let $d_i$ be the degree
of vertex $i$ of $H$, and $\rho(H)$ be the spectral radius of $H$. Then
\[
\rho(H)\geq
\left(\frac1n\sum_{i=1}^nd^{\frac{r}{r-1}}_i\right)^{\frac{r-1}{r}}.
\]
If $H$ is connected and $r\geq 3 $, then the equality holds if and only if $H$ is regular.
\end{lemma}

\noindent{\bfseries Proof of \autoref{thm:Main result-2}.}
We first prove the left hand. For short, denote $\rho(H)=\rho$. By AM-GM inequality, we have
\begin{align}\label{eq:AM-MB}
\frac{1}{r\!-\!1}\rho^{\frac{r}{r-1}}\!+\!\left(\frac{rm}{n}\right)^{\frac{r}{r-1}} \!&\! =
\frac{1}{r\!-\!1}\Bigg[\rho^{\frac{r}{r-1}}+\underbrace{\left(\frac{rm}{n}\right)^{\frac{r}{r-1}}+
\cdots+\left(\frac{rm}{n}\right)^{\frac{r}{r-1}}}_{r-1}\Bigg]\nonumber\\
& \geq\frac{r^2m}{(r-1)n}\rho^{\frac{1}{r-1}}.
\end{align}
Therefore, by \autoref{lem:rho bound concerning degrees} and \eqref{eq:AM-MB}, we have
\begin{align*}
v(H) & =\frac{1}{n}\sum_{i=1}^nd_i^{\frac{r}{r-1}}-\left(\frac{rm}{n}\right)^{\frac{r}{r-1}}\\
& \leq\rho^{\frac{r}{r-1}}-\left(\frac{rm}{n}\right)^{\frac{r}{r-1}}\\
& =\frac{r}{r-1}\rho^{\frac{r}{r-1}}-
\left[\frac{1}{r\!-\!1}\rho^{\frac{r}{r-1}}\!+\!\left(\frac{rm}{n}\right)^{\frac{r}{r-1}}\right]\\
& \leq\frac{r}{r-1}\rho^{\frac{1}{r-1}}\left(\rho-\frac{rm}{n}\right).
\end{align*}
Notice that $\rho\leq\frac{r}{\sqrt[r]{r!}}m^{\frac{r-1}{r}}$ by \autoref{lem:rho<m}. Hence
\[
v(H)\leq\frac{r}{r-1}\left(\frac{r}{\sqrt[r]{r!}}\right)^{\frac{1}{r-1}}m^{\frac{1}{r}}
\left(\rho-\frac{rm}{n}\right).
\]

Now we prove the right hand. Denote by $K_r^r$ the $r$-uniform hypergraph of order $r$ consisting
of a single edge. Let $\widetilde{H}$ be the direct product of $H$ and $K_r^r$,
i.e., $\widetilde{H}=H\times K_r^r$. Clearly, $\widetilde{H}$ is an $r$-partite $r$-uniform
hypergraph with partition
\[
V(\widetilde{H})=\bigcup_{j=1}^r\left(V(H)\times\{j\}\right).
\]
By \autoref{thm:Main result-1} we have
\begin{equation}
\label{eq:rho(tilde{H})-rm/n}
\rho(\widetilde{H})-\frac{|E(\widetilde{H})|}{n}\leq
\left(\frac{s_r(\widetilde{H})}{2}\right)^{\frac{r-1}{r}}.
\end{equation}
Notice that $|E(\widetilde{H})|=r!m$ and $d_{\widetilde{H}}(i,j)=(r-1)!d_i$ for any
$i\in V(H)$, $j\in[r]$. Therefore
\begin{align*}
s_r(\widetilde{H}) & =\sum_{j=1}^r\sum_{(i,j)\in V(\widetilde{H})}
\left|d_{\widetilde{H}}((i,j))-\frac{r!m}{rn}\right|\\
& =r\sum_{i\in V(H)}\left|(r-1)!d_{H}(i)-\frac{r!m}{rn}\right|\\
& =r!s(H).
\end{align*}
By \cite[Claim 4]{Liu:Bounds concerning degrees}, we know $\rho(\widetilde{H})=(r-1)!\rho(H)$.
It follows from \eqref{eq:rho(tilde{H})-rm/n} that
\[
(r-1)!\rho(H)-\frac{r!m}{n}\leq\left(\frac{r!s(H)}{2}\right)^{\frac{r-1}{r}},
\]
and the assertion follows by simple algebra. \qed

\end{document}